\documentclass[a4paper, reqno]{amsart}
\usepackage{mathpazo}
\usepackage{cite}
\usepackage{enumerate}
\usepackage{mathrsfs}

\newtheorem{theorem}{Theorem}
\newtheorem{lemma}[theorem]{Lemma}

\theoremstyle{remark}

\newtheorem{remark}[theorem]{Remark}

\begin{document}

\bibliographystyle{short}

\def\aM{M}
\def\Mat{{\mathbb{M}}}
\def\Z{{\mathbb{Z}}}
\def\N{{\mathbb{N}}}
\def\Rl{{\mathbb{R}}}
\def\Cx{{\mathbb{C}}}
\def\Lip{{\mathrm{Lip}\,}}

\def\dint{\mathop{\int\kern-6pt\int}}

\baselineskip=1.5\baselineskip

\title{Operator-Lipschitz functions in Schatten-von Neumann classes}
\author{Denis Potapov}
\email{d.potapov@unsw.edu.au}
\thanks{Research is partially supported by ARC}
\thanks{Corresponding author: f.sukochev@unsw.edu.au}
\author{Fedor Sukochev}
\email{f.sukochev@unsw.edu.au}
\address{School of Mathematics \& Statistics, University of NSW,
  Kensington NSW 2052 AUSTRALIA}


\maketitle

\begin{abstract}
  This paper resolves a number of problems in the perturbation theory
  of linear operators, linked with the 45 years old conjecture of
  M.G. Krein.  In particular, we prove that every Lipschitz function
  is operator Lipschitz in the Schatten-von Neumann ideals~$S^\alpha$,
  $1 < \alpha < \infty$. Alternatively, for every~$1 < \alpha <
  \infty$, there is a constant~$c _\alpha > 0$ such that $$ \left\|
    f(a) - f(b) \right\|_\alpha \leq \, c_\alpha \, \left\| f
  \right\|_{\Lip 1} \, \left\| a - b \right\|_\alpha, $$ where~$f$ is
  a Lipschitz function with $$ \left\| f \right\|_{\Lip 1} : =
  \sup_{\lambda, \mu \in \Rl \atop \lambda \neq \mu}\left| \frac
    {f(\lambda) - f(\mu)}{\lambda - \mu} \right| < + \infty $$ and
  where~$\left\| \cdot \right\|_\alpha$ is the norm of~$S^\alpha$
  and~$a, b$ are self-adjoint linear operators such that~$a - b \in
  S^\alpha$.
\end{abstract}

Denote by~$F_\alpha$ the class of functions~$f: \Rl \mapsto \Cx$ such
that $$ f(a) - f(b) \in S^\alpha $$ for any self-adjoint~$a, b$ such
that~$a -b \in S^\alpha$ and put $$ \left\| f \right\|_{F_\alpha} :=
\sup_{a, b} \frac {\left\| f(a) - f(b) \right\|_\alpha}{\left\| a -b
  \right\|_\alpha}. $$ In~\cite{KreinPerturbation1964}, M.G.~Krein
conjectured that the condition~$f' \in L^\infty$ is sufficient for~$f
\in F_1$.  The fact that this conjecture does not hold was established
by Yu.~Farforovskaya in~1972, see~\cite{Farforovskaya1972}.  Earlier,
in~1967 she had also established that the analogue of Krein's
conjecture does not hold for the case~$\alpha=\infty$,
see~\cite{Farforovskaja1968,Farforovskaja1967}.  Later,
T.~Kato~\cite{Kato1973} (respectively, E.B.~Davies,~\cite{Davies1988})
showed that already the absolute value function ~$f(t) = \left| t
\right|$ does not belong to~$F_\infty$ (respectively,~$F_1$).
Alternative approach to the study of class~$F_1$ was developed by
V.~Peller in~\cite{Pe1985}, where he proved that~$B^{1}_{\infty 1}
\subseteq F_1$ and every~$f \in F_1$ belongs to~$B^1_{11}$ locally,
where~$B^{s}_{pq}$ are Besov spaces; this approach also shows that the
class~$F_1$ is properly contained in the class of all Lipschitz
functions, since~$\Lip 1 \not \subseteq B^1_{11}$.  However,
the question whether the class~$F_\alpha$ contains all Lipschitz
functions when~$1<\alpha<\infty$, $\alpha \neq 2$ remained open.  For
example, $f(t) =\left| t \right|$ belongs to~$F_\alpha$ for
all~$1<\alpha<\infty$, see, e.g.,~\cite{Davies1988,
  KosakiLipCont1992-MR1152759, DDPS1997}.  Another sufficient
condition ensuring that~$f \in F_\alpha$, for every~$1 < \alpha <
\infty$ was obtained in~\cite{Davies1988, PSW-DOI} requires that
derivative~$f'$ is bounded and has bounded total variation.  On the
other hand, addressing the problem of description of the
classes~$F_\alpha$, $1<\alpha<\infty$, $\alpha \neq 2$, in her~1974
paper~\cite{Farforovskaya1974}, Yu.~Farforovskaya observes that ``a
slight change in the proof given in ~\cite{Farforovskaya1972}'' leads
to the existence of a Lipschitz function~$f_\alpha$ which does not
belong to~$F_\alpha$ for every~$1<\alpha<2$ and further proceeds with
an explicit ``example'' of~$f_\alpha$ such that~$f' \in L^\infty$,
but~$f \notin F_\alpha$, $2 < \alpha < \infty$.  Presenting the same
problem in~\cite{Peller1987}, V.~Peller conjectured that~$f \in
F_\alpha$, $1 \leq \alpha < 2$ implies that the lacunary Fourier
coefficients of~$f'$ satisfy $$ \left\{ \hat {f'}(2^n) \right\}_{n
  \geq 0} \in \ell^\alpha. $$

Finally, shortly before this paper was written, F.~Nazarov and
V.~Peller discovered that~$f(a) - f(b) \in S^{1, \infty}$,
provided~$f$ is Lipschitz and~$a, b$ are self-adjoint linear operators
such that~$a - b$ is one-dimensional, where~$S^{1, \infty}$ is the
weak trace ideal of compact operators.  From
here, they show also that $$ \left\| f(a) - f(b) \right\|_{\Omega}
\leq c_0\, \left\| a - b\right\|_1, $$ for every self-adjoint~$a, b$
such that~$a - b \in S^1$, where~$S^\Omega$ is dual to the Matsaev
ideal~$S^\omega$ (see~\cite{PelNaz2009} for details).

The main objective of the present paper is to show that in fact
M.G.~Krein's conjecture holds for all~$1 < \alpha < \infty$, that
is~$f' \in L^\infty$ implies~$\left\| f \right\|_{F_\alpha} <
\infty$. Equivalently, $F_\alpha$, $1<\alpha<\infty$ coincides with
the class of all Lipschitz functions. In particular, this shows that
Farforovskaya's remark concerning~$F_\alpha$, $1<\alpha<2$ and her
result for~$F_\alpha$, $2<\alpha<\infty$ given
in~\cite{Farforovskaya1974} do not hold and that the conjecture of
V.~Peller~\cite{Peller1987} does not hold either. 

The main result of the paper is the following theorem whose proof is
based on Theorems~\ref{TheoremOfTheCentury}.

\begin{theorem}
  \label{LipschitzContCompact}
  Let~$f$ be a Lipschitz function and let~$\left\| f \right\|_{\Lip 1}
  \leq 1$.  For every~$1 < \alpha < \infty$ there is a
  constant~$c_\alpha > 0$ such that 
  \begin{equation}
    \label{LipschitzEst}
    \left\| f(a) - f(b)
    \right\|_\alpha \leq \, c_\alpha \, \left\| a - b\right\|_\alpha, 
  \end{equation}
  where~$a$ and~$b$ are self-adjoint (possibly unbounded) linear
  operators such that~$a - b \in S^\alpha$.
\end{theorem}

The proof of Theorem~\ref{LipschitzContCompact} is given at the end of
Section~\ref{sec:doi-dd}.

The symbol~$c_\alpha$ shall denote a positive numerical constant which
depends only on~$\alpha$, $1\leq \alpha \leq \infty$ and which may
vary from line to line or even within a line.

\section{Schur multipliers of divided differences}
\label{sec:discrete}

Although\footnote{The proof presented in this section was
  substantially simplified in comparison to the firstly circulated
  argument.  A similar simplification was observed independently by
  Mikael de la Salle, \cite{MikalDeLaSalle}.}  the principal result of
the paper is proved for the ideals of compact operators, in the
present section, we shall work in the setting of an arbitrary
semifinite von Neumann algebra.  This wider setting brings no
additional difficulties to our considerations but allows very succinct
notations.  Yet a reader unfamiliar with theory of semifinite von
Neumann algebras may think of the algebra of all bounded linear
operators on~$\ell^2$ equipped with the standard trace instead of the
couple~$(\aM, \tau)$ and of the Schatten-von Neumann ideals~$S^\alpha$
instead of the noncommutative spaces~$L^\alpha$.

Let~$\aM$ be a von Neumann algebra with normal semifinite faithful
trace~$\tau$.  Let~$L^\alpha$, $1 \leq \alpha \leq \infty$ be the
$L^p$-space with respect to the couple~$(\aM, \tau)$
(see~\cite{PiXu2003}).

Let~$\left( e_k \right)_{k \in \Z} \subseteq \aM$ be a sequence of
mutually orthogonal projections and let~$f: \Rl \mapsto \Cx$ be a
Lipschitz function.  We shall study the following linear operator 
\begin{equation}
  \label{SM-sym-def}
  Tx = \sum_{k, j \in \Z} \phi_{kj} e_k x e_j,\ \ \phi_{kj} = \frac
  {f(k) - f(j)}{k -j},\ k \neq j,\ \ \phi_{kk} = 0. 
\end{equation}
We keep fixed the sequence~$\left( e_k \right)_{k \in \Z}$, the
function~$f$ and the operator~$T$ in the present section.

\begin{theorem}
  \label{TheoremOfTheCentury}
  If\/~$\left\| f \right\|_{\Lip 1} \leq 1$, then the operator~$T$ is
  bounded on every space~$L^\alpha$, $1 < \alpha < \infty$.
\end{theorem}

The symbol~$c_\alpha$ shall denote a positive numerical constant which
depends only on~$\alpha$, $1\leq \alpha \leq \infty$ and which may
vary from line to line or even within a line.

\begin{proof}[Proof of Theorem~\ref{TheoremOfTheCentury}]
  Without loss of generality, we may assume that~$f(0) = 0$ and
  that~$f$ is real-valued.

  Let us fix~$x \in L^\alpha$ and~$y \in L^{\alpha'}$,
  where~$\alpha^{-1} + \alpha'^{-1} = 1$, $1 < \alpha, \alpha' <
  \infty$.  We shall prove that $$ \left| \tau \left( y Tx \right)
  \right| \leq c_\alpha \, \left\| x \right\|_\alpha \, \left\| y
  \right\|_{\alpha'}. $$ Recall that the triangular truncation is a
  bounded linear operator on~$L^\alpha$, $1< \alpha < \infty$
  (e.g.~\cite{DDPS1997}).  Thus, we may further assume that the
  operators~$x$ is upper-triangular and~$y$ is
  lower-triangular.\footnote{An element~$x \in \aM$ is called {\it
      upper-triangular\/} (with respect to the sequence~$\left( e_k
    \right)_{k \in \Z}$) if and only if~$e_k x e_j = 0$ for every~$k >
    j$.  It is called {\it lower-triangular\/} if and only if~$x^*$ is
    upper-triangular.}  For every element~$z \in \aM$, we set~$z_{kj}
  := e_k z e_j$ for brevity.  Now we can write
  \begin{equation}
    \label{DecompositionI}
    \tau\left( y Tx
    \right) = \sum_{k < j} \tau \left( y_{jk} \phi_{kj} x_{kj}
    \right). 
  \end{equation}

  Let us show that we also may assume that the function~$f$ takes only
  integral values at integral points.  Indeed, by setting~$a_k = f(k)
  - f(k-1)$, we have $$ \phi_{kj} = \frac {1}{j-k} \sum_{k < m \leq j}
  a_m,\ \ k < j. $$ Thus, we continue $$ \tau \left( y Tx \right) =
  \frac 1{j - k} \sum_{k < j} \tau \left( y_{jk} x_{kj} \right)\,
  \sum_{k < m \leq j} a_m = \sum_{m \in \Z} a_m \sum_{k < m \leq j}
  \frac {\tau \left( y_{jk} x_{kj} \right)}{j - k}. $$ Recall that we
  have to show $$ \left| \tau \left( y Tx \right) \right| \leq
  c_\alpha \, \left\| x \right\|_\alpha \, \left\| y
  \right\|_{\alpha'}, $$ for every sequence~$\left( a_m \right) \in
  \ell^\infty$ with~$\left\| \left( a_m \right)\right\|_\infty \leq
  1$.  From this, it is clear that it is sufficient to take~$a_m = \pm
  1$ and thus, the function~$f$ takes only integral values at integral
  points, since $$ f(k) = f(k) - f(0) = \sum_{1 \leq m \leq k} a_m. $$
  We also may assume that the function~$f$ is non-decreasing
  (otherwise, we take the function~$f_1(t) = f(t) + t$).  Thus, from
  now on we assume that~$f$ takes integral values at integral points,
  $f$ is non-decreasing and $$ 0 \leq f(k) - f(j) \leq 2 \left( k - j
  \right),\ \ j \leq k,\ j, k \in \Z. $$ According to
  Lemma~\ref{DecompositionLemma}, we have
  \begin{equation}
    \label{DecompositionII}
    \phi_{kj} =
    \int_\Rl g(s)\, (f(j) - f(k))^{is} \, \left( j - k \right)^{-is}\,
    ds,\ \ k < j 
  \end{equation}
  where~$g: \Rl \mapsto \Cx$ such that
    \begin{equation}
    \label{Gfunction}
    \int_{\Rl} \left| s \right|^n \, \left| g(s)  \right| \, ds < +
    \infty,\ \ n \geq 0.
  \end{equation}
  With this in mind, we now see from~(\ref{DecompositionI})
  and~(\ref{DecompositionII}) $$ \tau \left( y Tx \right) = \int_\Rl
  g(s)\, \tau \left( y_s x_s \right)\, ds , $$ where (as in
  Lemma~\ref{MarcCorl}) $$ y_s = \sum_{k < j} \left( f(j) - f(k)
  \right)^{is} y_{jk} \ \ \text{and}\ \ x_s = \sum_{k < j} (j -
  k)^{-is} x_{kj}. $$

  Now it follows from Lemma~\ref{MarcCorl} that $$ \left| \tau \left(
      y_s x_s \right) \right| \leq \, c_\alpha \, \left( 1 + \left| s
    \right| \right)^2 \, \left\| x\right\|_\alpha \, \left\|
    y\right\|_{\alpha'} $$ and therefore, from~(\ref{Gfunction}),
   $$ \left| \tau \left( y Tx \right) \right| \leq c_\alpha \, \left\| x
   \right\|_\alpha \, \left\| y \right\|_{\alpha'}\, \int_\Rl \left( 1
     + \left| s \right| \right)^2 \, \left| g(s) \right| \, ds \leq
   c_\alpha \, \left\| x \right\|_\alpha \, \left\|
     y\right\|_\alpha. $$

\end{proof}

\begin{remark}
\label{TheoremOfTheCenturyRem}
The operator~$T$ in~(\ref{SM-sym-def}) can be also defined with
respect to two families of orthogonal projections: if~$\left( e_k
\right)_{k \in \Z}$ and~$\left( f_j \right)_{j \in \Z}$ are two
families of orthogonal projections, then
\begin{equation}
  \label{SM-def}
  \tilde Tx = \sum_{k, j \in \Z} \phi_{kj} e_k x f_j, 
\end{equation}
where~$\phi_{kj}$ as in~(\ref{SM-sym-def}).  In this case, the
operator~$\tilde T$ is also bounded on every~$L^\alpha$, $1 < \alpha <
\infty$ provided~$\left\| f \right\|_{\Lip 1} \leq 1$.  Indeed, to see
the latter it is sufficient to consider operator~$T$ as
in~(\ref{SM-sym-def}) with respect to the family of orthogonal
projections $$ \left\{ e_{j} \otimes e_{11} + f_j \otimes e_{22}
\right\}_{j \in \Z} $$ in the tensor product von Neumann algebra~$M
\otimes \Mat_2$, where~$\Mat_2$ is the algebra of $2 \times
2$-matrices, and apply Theorem~\ref{TheoremOfTheCentury}.
\end{remark}

To prove Lemma~\ref{MarcCorl} used in the proof of
Theorem~\ref{TheoremOfTheCentury}, we firstly need the following
result whose proof rather quickly follows from the vector-valued
Marcinkiewicz multiplier theorem.

\begin{theorem}
  \label{MarcThm}
  Let\/~$\lambda = \left( \lambda (n) \right)_{n \in \Z}$ be a
  sequence of complex numbers such that $$ \sup_{n \in \Z} \left|
    \lambda (n) \right| \leq 1. $$ If total variation of\/~$\lambda$
  over every dyadic interval~$2^k \leq \left| n \right| \leq 2^{k +
    1}$, $k \geq 0$ does not exceed~$1$, then the linear operator~$S$
  defined by $$ Sx = \sum_{k, j \in \Z} \lambda (f(j) - f(k))\, e_k x
  e_j,\ \ x \in L^\alpha $$ is bounded on every~$L^\alpha$, $1 <
  \alpha < \infty$, where~$f: \Z \mapsto \Z$ is any non-decreasing
  integral valued function.
\end{theorem}

\begin{proof}[Proof of Theorem~\ref{MarcThm}]
  It was proved in~\cite{Bo1986} that if~$X$ is a Banach space with
  UMD property (see~\cite{PiXu2003} for the relevant definitions) and
  if~$h \in L^2([0, 1], X)$ (= the space of all Bochner square
  integrable functions on~$[0, 1]$ with values in~$X$), then the
  linear operator\footnote{Here~$\left\{\hat h (n) \right\}_{n \in
      \Z}$ is the sequence of Fourier coefficients, i.e., $$ \hat h(n)
    = \int_0^1 h(t) \, e^{- 2 \pi i n t} \, dt,\ \ n \in \Z. $$} $$ M
  h (t) = \sum_{n \in \Z} \lambda(n) \hat h(n)\, e^{2\pi in t},\ \ t
  \in [0, 1] $$ is bounded provided $$ \sup_{n \in \Z} \left|
    \lambda(n) \right| \leq 1 $$ and the total variation of the
  sequence~$\lambda$ over every dyadic interval does not exceed~$1$.

  Recall that~$L^\alpha$ is a Banach space with UMD property for
  every~$1 < \alpha < \infty$ (see e.g.~\cite{PiXu2003}).  Consider
  the function $$ h_x(t) = u_t^* x u_t = \sum_{k, j \in \Z} e^{2 \pi i
    (f(j) - f(k)) t} \, e_k x e_j,\ \ x \in L^\alpha,\ t \in [0,
  1], $$ where the unitary~$u_t$ is defined by $$ u_t = \sum_{k \in
    \Z} e^{2 \pi i f(k) t} e_k. $$ Observe that the $n$-th Fourier
  coefficient of~$h_x$ is 
  \begin{equation}
    \label{HxFC}
    \hat h_x (n) = \sum_{f(j) - f(k) = n} e_k
    x e_j,\ \ n \in \Z. 
  \end{equation}
  Noting that the mapping~$x \in L^\alpha \mapsto h_x \in L^2([0, 1],
  L^\alpha)$ is a complemented isometric embedding of~$L^\alpha$
  into~$L^2([0, 1], L^\alpha)$ and that, from~(\ref{HxFC}) $$ M({h_x})
  = h_{Sx}, $$ we see that the boundedness of~$S$ on~$L^\alpha$, $1 <
  \alpha < \infty$ follows from that of~$M$ on~$L^2([0, 1],
  L^\alpha)$.
\end{proof}

\begin{lemma}
  \label{MarcCorl}
  If~$x \in L^\alpha$ and if~$$ x_s = \sum_{k < j} (f(j) - f(k))^{is}
  \, e_k x e_j,\ \ s \in \Rl, $$ then, for every~$1 < \alpha <
  \infty$, there is a constant~$c_\alpha > 0$ such that $$ \left\| x_s
  \right\|_\alpha \leq \, c_\alpha \, \left( 1 + \left| s \right|
  \right)\, \left\| x \right\|_\alpha, $$ where~$f: \Z \mapsto \Z$ is
  any non-decreasing integral function.
\end{lemma}

\begin{proof}[Proof of Lemma~\ref{MarcCorl}]
  Clearly, the lemma follows from Theorem~\ref{MarcThm} if we
  estimate the total variation of the sequence~$\lambda =
  \left\{n^{is} \right\}_{n > 0}$ over dyadic intervals.  To this end,
  via the fundamental theorem of the calculus, we see that $$ \left|
    n^{is} - (n + 1)^{is} \right| \leq \frac {\left| s \right|}{n},\ \
  n \geq 1 $$ and thus immediately $$ \sum_{2^k \leq n \leq 2^{k + 1}}
  \left| n^{is} - (n + 1)^{is} \right| \leq \, \left| s \right|,\ \ k
  \geq 0. $$ The lemma is proved.
\end{proof}

\begin{lemma}
  \label{DecompositionLemma}
  There is a function~$g: \Rl \mapsto \Cx$ such that $$ \int_{\Rl}
  \left| s \right|^n \left| g(s) \right| \, ds < + \infty,\ \ n \geq
  0 $$ and such that, for every~$\mu, \lambda > 0$ with~$0 \leq \frac
  \lambda \mu \leq 2$, $$ \frac \lambda \mu = \int_{\Rl} g(s)\,
  \lambda^{is} \mu^{-is}\, ds. $$
\end{lemma}

\begin{proof}[Proof of Lemma~\ref{DecompositionLemma}]
  Let us consider a $C^\infty$-function~$f$ such that (i)~$f \geq 0$,
  (ii)~$f (t) = 0$, if~$t \geq 1 + \log 2$; (iii)~$f(t) = e^{t}$,
  if~$t \leq \log 2$.  Observe that~$f$ and all its derivatives
  are~$L^2$ functions, i.e., $$ \left\| f^{(n)} \right\|_2 < +
  \infty,\ \ n \geq 0. $$ If we now set~$g(s) = \hat f(s)$,
  where~$\hat f$ is the Fourier transform of~$f$, then it is known
  (see~\cite[Lemma~7]{PoSuNGapps}) that $$ \int_\Rl \left| s \right|^n
  \, \left| g (s) \right| \, ds \leq c_0 \, \max \left\{\left\|
      f^{(n)} \right\|_2, \left\| f^{(n+1)} \right\|_2 \right\} < +
  \infty,\ \ n \geq 0. $$ Furthermore, via inverse Fourier transform,
  we also have $$ e^{t} = \int_{\Rl} g(s) \, e^{i t s} \, ds,\ \ t
  \leq 0. $$ and substituting~$t = \log \frac \lambda \mu$ delivers
  the desired relation.  The lemma is completely proved.
\end{proof}

\section{Double operator integrals of divided differences.}
\label{sec:doi-dd}

Here we present the continuous version of
transformation~(\ref{SM-def}) and Theorem~\ref{TheoremOfTheCentury}.
The proof of the continuous case employs
Theorem~\ref{TheoremOfTheCentury} (and
Remark~\ref{TheoremOfTheCenturyRem}) and some approximation procedure.

Let~$\aM$ be a semifinite von Neumann algebra with a normal semifinite
faithful trace~$\tau$ and~$L^\alpha$ be the corresponding $L^p$-space
with respect to the couple~$(\aM, \tau)$, $1 \leq \alpha \leq \infty$.
Let us first introduce a continuous version of
transformation~(\ref{SM-def}).

Let us fix spectral measures~$dE_\lambda, dF_\mu \in \aM$, $\lambda,
\mu \in \Rl$.  Recall that if~$x, y \in L^2$, then the mapping $$
(\lambda, \mu) \mapsto d \nu_{x, y} (\lambda, \mu) := \tau \left( y
  dE_\lambda x dF_\mu \right),\ \ \lambda, \mu \in \Rl $$ is a
complex-valued $\sigma$-additive measure on~$\Rl^2$ with total
variation not exceeding~$\left\| x \right\|_2 \, \left\| y \right\|_2$
(see~\cite{PSW-DOI}).  If now~$\phi: \Rl^2 \mapsto \Cx$ is a bounded
Borel function, then the latter implies that $$ (x, y) \mapsto
\int_{\Rl^2} \phi(\lambda, \mu)\, d \nu_{x, y} (\lambda, \mu),\ \ x, y
\in L^2 $$ is a continuous bilinear form, i.e., there is a bounded
linear operator~$T_\phi$ on~$L^2$ such that 
\begin{equation}
  \label{DOIdefBF}
  \tau \left( y T_\phi
    (x) \right) = \int_{\Rl^2} \phi\, d\nu_{x, y},\ \ x, y \in L^2. 
\end{equation}
The operator~$T_\phi$ is a continuous version of~(\ref{SM-def}), we
shall study in the present section.

We shall be saying that operator~$T_\phi$ is bounded on~$L^\alpha$ for
some~$1 \leq \alpha \leq \infty$, if and only if the operator~$T_\phi$
introduced above admits a bounded extension from~$L^\alpha \cap L^2$
to~$L^\alpha$.  The latter extension, if exists, is unique
(see~\cite{PSW-DOI}).

The following theorem is the main objective of the present section.

\begin{theorem}
  \label{TheoremOfTheCenturyCont}
  If~\/$\left\| f \right\|_{\Lip 1} \leq 1$, then the
  operator~$T_{\phi_f}$ is bounded on every space~$L^\alpha$, $1 <
  \alpha < \infty$, where $$ \phi_f (\lambda, \mu) = \frac {f(\lambda)
    - f(\mu)}{\lambda - \mu},\ \text{if~$\lambda \neq \mu$} \ \
  \text{and}\ \ \phi(\lambda, \lambda) = 0,\ \ \lambda, \mu \in
  \Rl. $$
\end{theorem}

We need the following auxiliary lemma first.

\begin{lemma}[{Duhamel's formula, \cite{Widom1984}}]
  \label{ExpEst}
  Let~$A$ and~$B$ be self-adjoint linear operators.  If~\/$A - B$ is
  bounded, then $$ e^{irA} - e^{irB} = ir \int_0^1 e^{ir (1 - t)A} \,
  \left( A - B \right)\, e^{irtB}\, dt,\ \ r \in \Rl. $$ In
  particular, $$ \left\| e^{irA} - e^{irB} \right\| \leq \left| r
  \right|\, \left\| A - B \right\|. $$
\end{lemma}


\begin{proof}[Proof of Theorem~\ref{TheoremOfTheCenturyCont}]
  We observe that function~$f$ may be taken finitely supported.
  Indeed, Let Theorem~\ref{TheoremOfTheCenturyCont} be proved for
  finitely supported functions.  Fix~$f: \Rl \mapsto \Cx$ Lipschitz,
  take a sequence of indicators~$\chi_n = \chi_{[-n, n]}$ and set $$
  \phi_n (\lambda, \mu) = \chi_n (\lambda)\, \phi_f \left( \lambda,
    \mu \right) \, \chi_n (\mu). $$ Since
  Theorem~\ref{TheoremOfTheCenturyCont} is proved for every finitely
  supported Lipschitz function, the sequence of operators~$T_{\phi_n}$
  is uniformly bounded, we also have that~$\lim_{n \rightarrow \infty}
  \phi_n = \phi_f$ pointwise.  This implies that~$T_{\phi_f}$ is also
  bounded (see~\cite[Lemma~2.6]{PSW-DOI}).

  Let us next assume that~$x \in L^2$ is {\it diagonal\/} with
  respect to the measures~$dE_\lambda$ and~$dF_\mu$, i.e., $$
  dE_\lambda x = x dF_\lambda,\ \lambda \in \Rl \ \
  \Longleftrightarrow\ \ x = \int_{\Rl} dE_\lambda x dF_\lambda. $$ In
  this case, $$ d \nu_{x, y} (\lambda, \mu) = \delta(\lambda - \mu)\,
  \tau \left( dE_\lambda x dF_\mu y \right), \ \ y \in L^2, $$
  where~$\delta$ is the Dirac function and therefore
  from~(\ref{DOIdefBF})
  \begin{multline*}
    \left| \tau(y T_\phi (x)) \right| = \left| \tau \left[ y
        \int_{\Rl} \phi(\lambda, \lambda)\, dE_\lambda x dF_\lambda
      \right] \right| \\ \leq \left\| y \right\|_{\alpha'}\, \left\|
      \int_\Rl \phi(\lambda, \lambda)\, dE_\lambda x dF_\lambda
    \right\|_\alpha \\ \leq \, \left\| y\right\|_{\alpha'} \, \left\|
      \phi \right\|_\infty \left\| \int_{\Rl} dE_\lambda x dF_\lambda
    \right\|_\alpha = \left\| \phi \right\|_\infty\, \left\| y
    \right\|_{\alpha'}\, \left\| x \right\|_\alpha.
  \end{multline*}
  for every $$ x \in L^2 \cap L^\alpha,\ \ y \in L^2 \cap L^{\alpha'}
  \ \ \text{and}\ \ \text{$x$ is diagonal}. $$ In other words, the
  operator~$T_\phi$ is trivially bounded on the diagonal part
  of~$L^\alpha$ for every bounded Borel function~$\phi$.  This
  indicates that the essential part of the proof is boundedness on the
  off-diagonal part of~$L^\alpha$, i.e., for those~$x \in L^\alpha$
  such that $$ dE_\lambda x dF_\lambda = 0,\ \lambda \in \Rl\ \
  \Longleftrightarrow\ \ \int_{\Rl} dE_\lambda x dF_\lambda = 0. $$
  Note also that if~$x = A - B$, then trivially $$ dE^A_\lambda x
  dE^B_\lambda = 0,\ \ \lambda \in \Rl. $$
  
  From this moment on we assume that~$x \in L^2$ is off-diagonal with
  respect to the the measures~$dE_\lambda$ and~$dF_\mu$.  The proof is
  a two stage approximation.  Let~$\left\{G_n\right\}_{n = 1}^\infty$
  be the family of dilated Gaussian functions as in
  Lemma~\ref{LiApprox}.  If now~$f_n = G_n * f$, then $$
  \phi_n(\lambda, \mu) := \phi_{f_n} (\lambda, \mu) = \int_{\Rl}
  G_n(s)\, \phi_f(\lambda - s, \mu - s)\, ds,\ \ \lambda \neq \mu.  $$
  In other words,~$\phi_n$ is convolution (with respect to the
  diagonal shift on~$\Rl^2$) of~$G_n$ and~$\phi$.  From
  Lemma~\ref{LiApprox}, we see that $$ \lim_{n \rightarrow \infty}
  \int_{\Rl^2} \left| \phi_n - \phi_f \right|\, d \nu_{x, y} = 0,\ \
  x, y \in L^2. $$ This implies that
  \begin{equation}
    \label{TheoremOfTheCenturyContTmp}
    \lim_{n \rightarrow \infty}
    \tau \left( y T_{\phi_n} (x) \right) = \tau \left( y T_{\phi_f} (x)
    \right),\ \ x, y \in L^2. 
  \end{equation}
  Consequently, in order to see that~$T_{\phi_f}$ is bounded
  on~$L^\alpha$, we have to prove that the operators~$T_{\phi_n}$ are
  all bounded on~$L^\alpha$, uniformly with respect to~$n = 1, 2,
  \ldots\,$.
  
  Observe that~$f_n$ is a rapidly decreasing $C^\infty$-function for
  every~$n = 1, 2, \ldots\,$.  Thus, if~$\hat f_n$ is the Fourier
  transform of~$f_n$, then 
  \begin{equation}
    \label{SummabilityA}
    c_{n, m} := \int_{\Rl} \left|s^m \hat f_n (s)
    \right|\, ds < + \infty,\ \ m = 0, 1, \ldots,\ n = 1, 2,
    \ldots\,. 
  \end{equation}
  Furthermore, if~$h_n(s)$ is the Fourier transform of
  derivative~$f'_n$, i.e, $h_n(s) = s \hat f(s)$, then~$h_n$ is
  integrable and
  \begin{multline}
    \label{PHInDEF}
    \phi_n (\lambda, \mu) = \phi_{f_n} (\lambda, \mu) = \frac {f_n
      (\lambda) - f_n(\mu)}{\lambda - \mu} = \int_0^1 f'_n \left( (1 -
      t) \lambda + t \mu\right)\, dt \\ = \int_\Rl h_n(s)\, ds
    \int_0^1 e^{is \left( (1 - t) \lambda + t \mu \right)}\, dt,\ \
    \lambda \neq \mu.
  \end{multline}
  Directly checking~(\ref{DOIdefBF}) again, the latter implies that
  the double operator integral operator~$T_{\phi_n}$ (with respect to
  any spectral measures~$dE_\lambda, dF_\mu \in \aM$, $\lambda, \mu
  \in \Rl$) has the form
  \begin{multline}
    \label{TensorRep}
    T_{\phi_n} (x) = \int_{\Rl} h_n(s)\, ds \int_0^1 e^{is (1 - t) A}
    x e^{is t B}\, dt, \\ \text{where}\ \ A = \int_{\Rl} \lambda \,
    dE_\lambda \ \ \text{and}\ \ B = \int_{\Rl} \mu \, dF_\mu \\
    \text{and~$x$ is off-diagonal.}
  \end{multline}
  Let us now fix the spectral measures~$dE_\lambda$ and~$dF_\mu$ and
  introduce the following discrete approximation~$dE_{m, \lambda}$
  and~$dF_{m, \mu}$, $\lambda, \mu \in \Rl$, $m = 1, 2, \ldots$ 
  \begin{multline*}
    E_{m} (\Omega) = \sum_{j \in \Z:\ \frac jm \in \Omega} e_j \ \
    \text{and}\ \ F_{m} (\Omega) = \sum_{k \in \Z:\ \frac km \in
      \Omega} f_k, \ \
    \text{where} \\
    e_j = E \left[\frac jm, \frac {j + 1}m \right) \ \ \text{and}\ \
    f_k = F \left[\frac km, \frac {k + 1}m \right),\\ j, k \in \Z,\ \
    \Omega \subseteq \Rl,\ \text{$\Omega$ is Borel}.
  \end{multline*}
  Let also~$T_{n, m}$ be the double operator integral operator
  associated with the function~$\phi_n$ and the spectral
  measures~$dE_{m, \lambda}$, $dF_{m, \mu}$.  On one hand then, the
  operator~$T_{n, m}$ has the form $$ T_{n, m} (x) = \sum_{j, k \in
    \Z} \phi_n\left( \frac jm, \frac km \right) e_j x f_k = \sum_{j, k
    \in \Z} \frac {mf_n \left( \frac jm \right) - m f_n \left( \frac
      km\right) }{j - k}\, e_j x f_k, $$ which is operator given
  in~(\ref{SM-def}) with respect to the function~$f_{n, m} (t) =m f_n
  (m^{-1} t)$.  Since~$f_{n, m}$ is Lipschitz and $$ \left\| f_{n, m}
  \right\|_{\Lip 1} = \left\| f_n \right\|_{\Lip 1} = \left\| G_n * f
  \right\|_{\Lip 1} \leq 1, $$ it follows from
  Theorem~\ref{TheoremOfTheCentury} that the family of
  operators~$T_{m,n}$ is bounded on~$L^\alpha$ uniformly across~$m, n
  = 1, 2, \ldots\,$.  On the other hand, let us set $$ A_m =
  \int_{\Rl} \lambda\, dE_{m, \lambda} \ \ \text{and}\ \ B_m =
  \int_\Rl \mu \, dF_{m, \mu}. $$ It follows immediately from the
  spectral theorem that $$ \left\| A - A_m \right\| \leq \frac 1m \ \
  \text{and}\ \ \left\| B - B_m \right\| \leq \frac 1m. $$
  Furthermore, observing the elementary representation
  \begin{multline*}
    e^{is (1 - t)A} x e^{is t B} - e^{is (1 - t) A_m} x e^{is B_m} =
    e^{is (1 - t) A} x \left( e^{is B} - e^{is B_m} \right) \\ +
    \left( e^{is (1 - t) A} - e^{is (1 - t) A_m} \right) x \, e^{is
      tB_m},
  \end{multline*}
  we now estimate, using Lemma~\ref{ExpEst}, $$ \left\| e^{is (1 - t)
      A} x e^{ist B} - e^{is (1 - t)A_m} x e^{-ist B_m}
  \right\|_\alpha \leq \frac {2\left| s \right|}{m}\, \left\| x
  \right\|_\alpha. $$ Employing~(\ref{TensorRep}) for
  operators~$T_{\phi_n}$ and~$T_{n, m}$, we have $$ T_{\phi_n} (x) -
  T_{n, m} (x) = \int_{\Rl} h_n (s) \, ds \int_0^1 \left[ e^{is(1-t)A}
    x e^{ist B} - e^{is (1-t)A_m} x e^{istB_m} \right]\, dt. $$ Using
  the estimate above for the integrand, we finally see $$ \left\|
    T_{\phi_n} (x) - T_{n, m} (x) \right\|_\alpha \leq \frac 2m\,
  \left\| x\right\|_\alpha \, \int_\Rl \left| s h_n (s)\right|\, ds
  = \frac {2 c_{n, 2}}m \, \left\| x \right\|_\alpha, $$ since
  function~$h_n(s) = s\hat f(s)$ and~(\ref{SummabilityA}).  In
  particular, we have $$ \lim_{m \rightarrow \infty} \left\|
    T_{\phi_n} (x) - T_{n, m} (x) \right\|_\alpha = 0. $$ We have
  already observed above that the family of operators
  operators~$T_{n,m}$ is uniformly bounded on~$L^\alpha$.  Together
  with the limit above, it means that the family~$T_{\phi_n}$ is also
  uniformly bounded on~$L^\alpha$.
  Recalling~(\ref{TheoremOfTheCenturyContTmp}) finally yields
  that~$T_{\phi_f}$ is bounded on~$L^\alpha$.  The theorem is proved.
\end{proof}

\begin{lemma}
  \label{LiApprox}
  Let~$d\nu$ be a finite Borel measure on~$\Rl^2$ and let~$\phi \in
  \Rl^2 \mapsto \Cx$ an integrable function on~$\left( \Rl^2, d\nu
  \right)$, i.e., $\phi \in L^1(\Rl^2, d\nu)$.  If~$G_n$ is the family
  of dilated Gaussian functions, i.e., $$ G_n(t) = nG(nt),\ \ G(t) =
  (2\pi)^{-\frac 12}\, e^{- \frac {t^2}2 },\ \ t \in \Rl,\ n = 1, 2,
  \ldots $$ and if $$ \phi_n (\lambda, \mu) = \int_{\Rl} G_n(s) \phi
  (\lambda - s, \mu - s)\, ds, $$ then
  \begin{equation}
    \label{LiApproxObj}
    \lim_{n \rightarrow \infty} \int_{\Rl^2} \left| \phi_n - \phi
    \right|\, d\nu = 0. 
  \end{equation}
\end{lemma}

\begin{proof}[Proof of Lemma~\ref{LiApprox}]
  Since the class of uniformly continuous functions on~$\Rl^2$ is norm
  dense in~$L^1(\Rl^2, d \nu)$, we may assume that~$\phi$ is uniformly
  continuous on~$\Rl^2$.  For continuous~$\phi$,
  limit~(\ref{LiApproxObj}) can proved by the standard approximation
  identity argument.  Indeed, we shall actually show the uniform
  convergence $$ \lim_{n \rightarrow \infty} \left\| \phi_n -
    \phi\right\|_\infty = 0, $$ which, since~$d \nu$ is finite,
  implies~(\ref{LiApproxObj}).

  Observe first that $$ \int_{\Rl} G_n (s)\, ds = 1,\ \ n = 1, 2,
  \ldots\,. $$ Consequently, we have
  \begin{multline*}
    \phi_n (\lambda, \mu) - \phi(\lambda, \mu) = \int_\Rl G_n (s) \,
    \phi(\lambda - s, \mu - s)\, ds - \phi(\lambda, \mu)\, \int_\Rl
    G_n (s)\, ds \\ = \int_\Rl \Phi_n (\lambda, \mu, s)\, ds,\\
    \text{where}\ \ \Phi_n(\lambda, \mu, s) = G_n (s)\, \left(
      \phi(\lambda - s, \mu - s) - \phi(\lambda, \mu) \right).
  \end{multline*}
  Fix~$\epsilon > 0$.  Recall that~$\phi$ is uniformly continuous
  which means that there is~$\delta > 0$ such that $$ \sup_{\lambda,
    \mu \in \Rl} \left| \phi(\lambda - s, \mu - s) - \phi(\lambda,
    \mu) \right| \leq \epsilon,\ \ \text{if~$\left| s \right| <
    \delta$.} $$ Using the latter~$\delta > 0$, we split 
  \begin{multline*}
    \phi_n(\lambda, \mu) - \phi(\lambda, \mu) = \omega_0(\lambda, \mu) +
    \omega_\infty (\lambda, \mu),\ \ \text{where} \\ \omega_0 (\lambda,
    \mu) = \int_{\left| s \right| < \delta} \Phi_n (\lambda, \mu, s)\,
    ds \ \ \text{and}\ \ \omega_\infty (\lambda, \mu) = \int_{\left| s
      \right| \geq \delta} \Phi_n (\lambda, \mu, s)\, ds. 
  \end{multline*}
  We estimate the latter two terms separately.  For~$\omega_0$, using
  the uniform estimate above, 
  \begin{multline*}
    \sup_{\lambda, \mu \in \Rl} \left| \omega_0 (\lambda, \mu) \right|
    \leq \int_{\left| s \right| < \delta} G_n (s)\, \sup_{\lambda, \mu
      \in \Rl} \left| \phi(\lambda - s, \mu - s) - \phi(\lambda, \mu)
    \right|\, ds \\ \leq \epsilon \, \int_{\Rl} G_n (s)\, ds. \leq
    \epsilon.
  \end{multline*}
  On the other hand, for~$\omega_\infty$, $$ \sup_{\lambda, \mu \in
    \Rl} \left| \omega_\infty (\lambda, \mu) \right| \leq 2 \,
  \int_{\left| s \right| \geq \delta} G_n (s)\, \sup_{\lambda, \mu}
  \left| \phi(\lambda, \mu) \right|\, ds = 2 \left\| \phi
  \right\|_\infty\, \int_{\left| s \right| \geq \delta} G_n (s)\,
  ds. $$ Noting that for every fixed~$\delta$, we have $$ \lim_{n
    \rightarrow \infty} \int_{\left| s \right| \geq \delta} G_n (s)\,
  ds = 0, $$ we conclude that there is~$N_\epsilon$ such that for
  every~$n \geq N_\epsilon$, $$ \sup_{\lambda, \mu \in \Rl} \left|
    \omega_\infty (\lambda, \mu) \right| \leq \epsilon. $$ Combining
  the estimates above for~$\omega_0$ and~$\omega_\infty$, we obtain
  that for every~$\epsilon > 0$, there is~$N_\epsilon$ such that for
  every~$n \geq N_\epsilon$, $$ \left\| \phi_n - \phi\right\|_\infty
  \leq \left\| \omega_0 \right\|_\infty + \left\| \omega_\infty
  \right\|_\infty \leq 2 \epsilon. $$ The lemma is proved.
\end{proof}

We are now ready to proof Theorem~\ref{LipschitzContCompact}.

\begin{proof}[Proof of Theorem~\ref{LipschitzContCompact}]
  Observe that it is sufficient to prove that there is a
  constant~$c_\alpha$ such that for every compact self-adjoint
  operator~$u$ and every bounded operator~$v$ 
  \begin{equation}
    \label{ComEst}
    \left\| [ f(u), v
      ]\right\|_\alpha \leq \, c_\alpha \, \left\| [u, v]
    \right\|_\alpha. 
  \end{equation}
  Indeed, estimate~(\ref{LipschitzEst}) immediately follows from the
  inequality above with
  \begin{equation*}
    \label{UandVoperators}
    u =
    \begin{pmatrix}
      a & 0 \\ 0 & b \\
    \end{pmatrix}
    \ \ \text{and}\ \  
    v = 
    \begin{pmatrix}
      0 & 1 \\ 1 & 0 \\
    \end{pmatrix}.
  \end{equation*}
  Let~$T_{\phi_f}$ be the double operator integral operator associated
  with the function~$\phi_f$ and spectral measures~$dE = dF = dE^u$,
  where~$dE^u_\lambda$ is the spectral measure of~$u$.  It follows
  from Theorem~\ref{TheoremOfTheCenturyCont} that
  operator~$T_{\phi_f}$ is bounded on every~$L^\alpha$, $1 < \alpha <
  \infty$ and its norm does not depend on operator~$u$.  Combining
  this fact with~\cite[Theorem~5.3]{PoSu} yields the
  estimate~(\ref{ComEst}).
\end{proof}

\bibliography{references.bib}

\providecommand{\bysame}{\leavevmode\hbox to3em{\hrulefill}\thinspace}
\begin{thebibliography}{10}

\bibitem{Bo1986}
J.~Bourgain, \emph{Vector-valued singular integrals and the {$H\sp 1$}-{BMO}
  duality}, Probability theory and harmonic analysis (Cleveland, Ohio, 1983),
  Monogr. Textbooks Pure Appl. Math., vol.~98, Dekker, New York, 1986,
  pp.~1--19.

\bibitem{Davies1988}
E.~B. Davies, \emph{Lipschitz continuity of functions of operators in the
  {S}chatten classes}, J. London Math. Soc. (2) \textbf{37} (1988), no.~1,
  148--157.

\bibitem{MikalDeLaSalle}
M.~de~la Salle, \emph{A shorter proof of a result by {P}otapov and {S}ukochev
  on {L}ipschitz functions on~${S}^p$}, ArXiv:0905.1055.

\bibitem{PSW-DOI}
B.~de~Pagter, F.~A. Sukochev, and H.~Witvliet, \emph{Double operator
  integrals}, J. Funct. Anal. \textbf{192} (2002), no.~1, 52--111.

\bibitem{DDPS1997}
P.~G. Dodds, T.~K. Dodds, B.~de~Pagter, and F.~A. Sukochev, \emph{Lipschitz
  continuity of the absolute value and {R}iesz projections in symmetric
  operator spaces}, J. Funct. Anal. \textbf{148} (1997), no.~1, 28--69.

\bibitem{Farforovskaja1967}
Y.~B. Farforovskaya, \emph{An estimate of the nearness of the spectral
  decompositions of self-adjoint operators in the {K}antorovi\v c-{R}ubin\v
  ste\u\i n metric}, Vestnik Leningrad. Univ. \textbf{22} (1967), no.~19,
  155--156.

\bibitem{Farforovskaja1968}
\bysame, \emph{The connection of the {K}antorovi\v c-{R}ubin\v ste\u\i n metric
  for spectral resolutions of selfadjoint operators with functions of
  operators}, Vestnik Leningrad. Univ. \textbf{23} (1968), no.~19, 94--97.

\bibitem{Farforovskaya1972}
\bysame, \emph{An example of a {L}ipschitz function of self-adjoint operators
  with non-nuclear difference under a nuclear perturbation.}, Zap. Nauchn. Sem.
  Leningrad. Otdel. Mat. Inst. Steklov. (LOMI) \textbf{30} (1972), 146--153.

\bibitem{Farforovskaya1974}
\bysame, \emph{Regarding the estimate of~$\|f(b) - f(a)\|$ in the
  clasess~$\mathfrak{S}_p$}, Zap. Nauchn. Sem. Leningrad. Otdel. Mat. Inst.
  Steklov. (LOMI) \textbf{39} (1974), 194--195.

\bibitem{Kato1973}
T.~Kato, \emph{Continuity of the map {$S\mapsto \mid S\mid $} for linear
  operators}, Proc. Japan Acad. \textbf{49} (1973), 157--160.

\bibitem{KosakiLipCont1992-MR1152759}
H.~Kosaki, \emph{Unitarily invariant norms under which the map {$A\to \vert
  A\vert $} is {L}ipschitz continuous}, Publ. Res. Inst. Math. Sci. \textbf{28}
  (1992), no.~2, 299--313.

\bibitem{KreinPerturbation1964}
M.~G. Kre{\u\i}n, \emph{Some new studies in the theory of perturbations of
  self-adjoint operators}, First {M}ath. {S}ummer {S}chool, {P}art {I}
  ({R}ussian), Izdat. ``Naukova Dumka'', Kiev, 1964, pp.~103--187.

\bibitem{PelNaz2009}
F.~Nazarov and V.~Peller, \emph{Lipschitz functions of perturbed operators}, C.
  R. Math. Acad. Sci. Paris \textbf{347} (2009), no.~15-16, 857--862.

\bibitem{Pe1985}
V.~V. Peller, \emph{Hankel operators in the theory of perturbations of unitary
  and selfadjoint operators}, Funktsional. Anal. i Prilozhen. \textbf{19}
  (1985), no.~2, 37--51, 96.

\bibitem{Peller1987}
\bysame, \emph{For which {$f$} does {$A-B\in S\sb p$} imply that {$f(A)-f(B)\in
  S\sb p$}?}, Operators in indefinite metric spaces, scattering theory and
  other topics (Bucharest, 1985), Oper. Theory Adv. Appl., vol.~24,
  Birkh\"auser, Basel, 1987, pp.~289--294.

\bibitem{PiXu2003}
G.~Pisier and Q.~Xu, \emph{Non-commutative {$L\sp p$}-spaces}, Handbook of the
  geometry of Banach spaces, Vol.\ 2, North-Holland, Amsterdam, 2003,
  pp.~1459--1517.

\bibitem{PoSuNGapps}
D.~Potapov and F.~Sukochev, \emph{Unbounded {F}redholm modules and double
  operator integrals}, J. Reine Angew. Math. \textbf{626} (2009), 159--185.

\bibitem{PoSu}
D.~Potapov and F.~Sukochev, \emph{Lipschitz and commutator estimates in
  symmetric operator spaces}, J. Operator Theory \textbf{59} (2008), no.~1,
  211--234.

\bibitem{Widom1984}
H.~Widom, \emph{When are differentiable functions differentiable?}, Linear and
  Complex Analysis Problem Book, Lecture Notes in Mathematics, vol. 1043, 1984,
  pp.~184--188.

\end{thebibliography}

\end{document}